\documentclass[a4paper,twoside,10pt]{article}
\usepackage{amsthm}
\newtheorem{theorem}{Theorem}[section]

\newtheorem{lemma}{Lemma}[theorem]
\newtheorem{definition}{Definition}

\usepackage{amssymb}                 
\usepackage{amsmath}

\usepackage{latexsym}                 

\usepackage{verbatim}                
\usepackage{graphicx}
\usepackage{epstopdf}
\usepackage{epsfig}              
\usepackage{color}                  
\usepackage[normal]{caption2}
\usepackage[utf8]{inputenc}
\usepackage[english]{babel}
\usepackage{multicol}
\usepackage{makeidx}
\usepackage[normal]{subfigure}
\usepackage{url}
\usepackage{hyperref}
\usepackage{mathtools}
\usepackage{plain}
\usepackage[utf8]{inputenc} 
\usepackage{mathtools}
\usepackage{fancyhdr}
\begin{document}
{\Large
\begin{center}
\bf{\LARGE An Analysis of the Generalized Gaussian Integrals and Gaussian Like Integrals of Type I and II}
\end{center}}
\begin{center}
Prakash Pant  $^{1}  $, Hem Lal Dhungana  $^{2}  $, Sudip Rokaya  $^{3}  $
\end{center}
\begin{center}
{\footnotesize
  $^{1}  $ The University of Vermont, Burlington, Vermont, USA \\
  $^{2}  $ Mid-West University, Surkhet, Nepal\\
  $^{3}  $ Massachusetts Institute of Technology, Cambridge, Massachusetts, USA
\\[1mm]
Correspondence to: Hem Lal Dhungana, Email: hem.dhungana@mu.edu.np
}
\end{center}
\noindent
\textbf{Abstract:} The Gaussian integral, denoted as \( \int_{-\infty}^{\infty} e^{-x^2} dx \), plays a significant role in mathematical literature. In this paper, we explore a family of integrals related to Gaussian functions. Specifically, we introduce generalized Gaussian integrals, represented as \( \int_{0}^{\infty} e^{-x^n} dx \), and two distinct types of Gaussian-like integrals: 

1. Type I: \( \int_{0}^{\infty} e^{-f(x)^2} dx \), and  
2. Type II: \( \int_{0}^{\infty} e^{-x^2} f(x) dx \),  \\
\noindent
where \( f(x) \) is a continuous function. The study of integrals related to Gaussian-like functions has been explored in the work  of Huang\cite{H2023} and Dominy \cite{Dnd}. Our approach to evaluating these integrals relies on specialized functions, including error functions, complementary error functions, imaginary error functions, and Basel functions.\\
All the authors contributed equally to deriving the results and preparing the manuscript.\\\\
2020\textit{ Mathematics Subject Classification :} 45R05 \\
\textbf{Keywords:} Euler-Mascheroni constant, Laurent series, Error functions, Fubini's Theorem
\\ \\

\section{Introduction}
The Gaussian integral, which is central to various fields of mathematical analysis, probability theory, and physical sciences, finds its historical roots in the work of Carl Friedrich Gauss. In \textit{Theoria Motus Corporum Coelestium in Sectionibus Conicis Solem Ambientium} (1809), Gauss introduced the least squares method and the concept of a normal distribution, setting the foundation for the Gaussian function's later applications in probability and statistics \cite{G1809}. Gauss extended his analysis of observational errors in Theoria Combinationis Observationum Erroribus Minimis Obnoxiae (1823), where he mathematically derived a model of error distribution in empirical data \cite{G1823}. The role of standard normal distribution  function \[ f(x)=\frac{1}{\sigma\sqrt{2\pi}} e^{-\frac{(x-\mu)^2}{2\sigma^2}} dx \] in probability and statistics can be found in the work of Laplace \cite{L1812} and Stiger \cite{S1986}. The normal distribution has been known to millions as a bell curve, bell-shaped curve or Gaussian distribution. It is not at all obvious at first that integral of the standard normal distribution function over real numbers is 1. Proofs of the integral of similar function   $\int_{-\infty}^{\infty} e^{-x^2} dx = \sqrt{\pi}   $ using techniques like polar transformations, differentiation under integral sign, gamma function can be found in the paper of Cornard \cite{C2016}. Solutions of Generalised Gaussian Integral of type II using dimensional analysis can be found in the work of Dominy \cite{Dnd}. Huang also gave several interesting proofs of various Gaussian Like Integrals of type-II using real and complex methods \cite{H2023} \footnote{Huang's paper motivated us to evaluate a slightly different version of Gaussian Integrals.}. Even though a lot of Type II generalizations of the form \( \int_0^{\infty} e^{-x^2} f(x) dx \) have been evaluated  in the mathematical literature, there is inadequacy of Type I generalization \( \int_0^{\infty} e^{-f(x)^2} dx \) for various functions \( f(x) \) . This study aims to fill the gap in evaluation of Gaussian Like Integrals of Type-II.  These evaluations are indispensable in numerous fields, as Gaussian integrals are fundamental in stochastic methods \cite{G2004}, Machine Learning \cite{R2006}, and in statistical mechanics for partition functions\cite{P2011}. The importance of Gaussian / Normal distribution in  hypothesis testing, estimation, and the central limit theorem, solidifying its role in both theoretical and applied statistics is also examined in the paper of R.A. Fisher\cite{F1922}.
 
 In this study of Gaussian Like Integrals, error function   $erf(x)= \frac{2}{\sqrt{\pi}}\int_0^x e^{-t^2} dt   $ and its counterparts, namely complementary error function   $erfc(x)=\frac{2}{\sqrt{\pi}} \int_x^{\infty} e^{-t^2} dt  $ and imaginary error function   $erfi(x)=\frac{2}{\sqrt{\pi}} \int_0^x e^{t^2} dt  $ have been carefully manipulated to present various fruitful results. 
 We organize the work into different sections. In section 2, we evaluate Generalized Gaussian Integrals. Several Gaussian Like Integrals of type-I and type-II have been evaluated in section 3 and section 4 respectively. In section 5, we highlight a series of miscellaneous integrals. As an additional check, all the formulas have been verified from Wolfram Alpha.

 \section{Generalized Gaussian Integral}
This section contains a number of generalizations associated with the well known Gaussian integral. We present our main results in terms of error functions \cite{AS1972} and Euler-Mascheroni constant \cite{Wei2002}. First, we mention some definitions, lemmas, and then we construct our main results.

\begin{definition}
From \cite{SG2002}, we define   $\Gamma(x)  $ for   $x>0  $ as following:
\begin{equation}\label{Gamma}
 \Gamma(x) = \int_0^{\infty} e^{-t} t^{x-1} dt 
\end{equation}
\end{definition}
Using integration by parts, we see that for positive integer $x$,   $\Gamma(x)=(x-1)!  $

\begin{lemma}
    For   $\Re(z)>0  $ , the Laurent series expansion of Gamma function is 
\begin{equation}\label{gammaseries}
     \Gamma(z)=\frac{1}{z}-\gamma +\frac{1}{2}(\gamma^2+\frac{\pi^2}{6})z-\frac{1}{6}(\gamma^3+\frac{\gamma \pi^2}{2}+2\zeta(3) )z^2+O(z^3)  
\end{equation}
where   $\gamma  $ represents the Euler-Mascheroni constant.
\end{lemma}
A proof for this series can be found on \cite{M2016}.

\begin{lemma}
For   $\Re(z)>0  $, as   $n \in \mathbb{N}  $ tends to infinity, we have:
\begin{equation}\label{approx}
 \Gamma \left(\frac{1}{n}\right) \sim n - \gamma   
\end{equation}
where   $\gamma  $ is Euler-Mascheroni constant \cite{Wei2002} and   $\sim  $ denotes asymptotic equivalence.
\end{lemma}
\begin{proof} Substituting $z =  \frac{1}{n}  $ and letting $n$ tend to infinity, we get the result.   $\,\,    $ \\
To verify this, using wolfram alpha, we get \\
   $\Gamma(\frac{1}{3}) \approx 2.7689 \,\, ,  $
  $\Gamma(\frac{1}{4}) \approx 3.6256 \,\, ,  $
  $\Gamma(\frac{1}{5}) \approx 4.5908 \,\, ,  $ \\
These approximations are reasonably close to the values obtained using \eqref{approx} with   $\gamma=0.5772  $. The accuracy of the approximation improves as $n$ becomes larger.
\end{proof}

\begin{theorem}
For   $n>0  $,
\begin{equation}\label{gen}
    \int_0^{\infty} e^{-x^n} dx = \frac{1}{n}\Gamma(\frac{1}{n}) \sim (1-\frac{\gamma}{n}) 
\end{equation}
\end{theorem}
\begin{proof} Making an u-substitution $u = x^n  $,
\[\int_0^{\infty} e^{-x^n} dx = \frac{1}{n}\int_0^{\infty} e^{-u} u^{\frac{1}{n}-1} du \xrightarrow{ Using \eqref{Gamma}} \frac{1}{n}\Gamma(\frac{1}{n}) \]
Applying this general result, we obtain the following values: \\
For n=1, we get   $\frac{\Gamma(1)}{1}  $ = 1 \\
For n=2, we get   $\frac{1}{2}\Gamma(\frac{1}{2}) =\frac{\sqrt{\pi}}{2}   $ \cite{C2016} \\
We do not have a closed form for   $n>2  $. However, we can use the asymptotic approximation for   $\Gamma(\frac{1}{n})$ from \eqref{approx} considering $n$ is large. This approximation validates the claim in \eqref{gen}.  $\,\,    $   
\end{proof}
\section{Gaussian Like Integral of Type-I}
In this section, we prove various results involving integrals in the form of   $\int_0^{\infty}e^{-f(x)^2} dx   $. The limits of integration have been modified in some cases according to the domain of a function. Most of the results are expressed in terms of the error functions: $erf(x)$, $erfc(x)$, and $erfi(x)$. We begin by providing clear definitions for these error functions, establishing relevant lemmas, and then proceeding to serve our main results. \\

In the existing literature, the definition of error functions has sometimes been ambiguous due to the inclusion or omission of   $\frac{\sqrt{\pi}}{2}  $ depending on the author's perspective\cite{EG1969}. To make the following clarity, we adopt the following definitions:
\begin{definition}
 
For a complex number x,
\begin{equation}\label{erf}
 \int_0^x e^{-t^2} dt = \frac{\sqrt{\pi}}{2}erf(x) 
\end{equation}
\begin{equation}\label{erfc}
 \int_x^{\infty} e^{-t^2} dt = \frac{\sqrt{\pi}}{2} erfc(x) 
\end{equation}
\begin{equation}\label{erfi} 
 \int_0^x e^{t^2} dt = \frac{\sqrt{\pi}}{2} erfi(x)  
\end{equation}
\end{definition}

We use these definitions without referencing them in the later sections. \\

We present some properties of error functions that will be helpful during the evaluation of the Gaussian-like Integrals. 
\begin{lemma}
For a complex number x,
\begin{equation}\label{irelation}
 erf(ix)=ierfi(x) 
\end{equation}
\begin{equation}\label{crelation}
 erf(x)+erfc(x)=1 
\end{equation}
\end{lemma}
\begin{proof}
\[ erf(ix) =\frac{2}{\sqrt{\pi}} \int_0^{ix} e^{-t^2} dt \xrightarrow{-it\rightarrow t}  \frac{2}{\sqrt{\pi}} \int_0^x e^{t^2} i dt  = ierfi(x)  \] which proves \eqref{irelation}. 
\[ erf(x)+erfc(x) =  \frac{2}{\sqrt{\pi}} \int_0^{x} e^{-t^2} dt +\frac{2}{\sqrt{\pi}} \int_x^{\infty} e^{-t^2} dt = \frac{2}{\sqrt{\pi}} \int_0^{\infty} e^{-t^2} dt = \frac{2}{\sqrt{\pi}} \times \frac{\sqrt{\pi}}{2} = 1    \]
which proves \eqref{crelation}. 
\end{proof}

\begin{lemma}
For a complex number x,
\begin{equation}\label{odderf}
 erf(-x)=-erf(x) 
\end{equation}
\begin{equation}\label{odderfi}
 erfi(-x)=-erfi(x) 
\end{equation}
\begin{equation}\label{negerfc}
 erfc(-x)=2-erfc(x) 
\end{equation}
\end{lemma}
\begin{proof}
\[ erf(-x) = \frac{2}{\sqrt{\pi}} \int_0^{-x} e^{-t^2} dt \xrightarrow{-t\rightarrow t} - \frac{2}{\sqrt{\pi}} \int_0^x e^{-t^2} dt  = -erf(x)  \] which proves lemma \eqref{odderf}. Proof of lemma $\eqref{odderfi}$ is similar. 
To prove lemma $\eqref{negerfc}$, we use lemma $\eqref{odderf}$ and lemma $\eqref{crelation}$. 
\[ erfc(-x) \xrightarrow{From\, \eqref{crelation}} 1- erf(-x) \xrightarrow{From\, \eqref{odderf}} 1+erf(x) \xrightarrow{From\, \eqref{crelation}} 2-erfc(x)  \,\,  \]
\end{proof}

\begin{lemma}

\begin{equation}\label{erfvalues}
 erf(\infty)=1; erf(0)=0; erf(-\infty)=-1 
\end{equation}
\begin{equation}\label{erfcvalues}
 erfc(0)=1 ; erfc(\infty)=0 ; erfc(-\infty)=2 
\end{equation}
\begin{equation}\label{erfivalues}
 erfi(0)=0 
\end{equation}
\end{lemma}
\begin{proof}
\[ erf(\infty) = \frac{2}{\sqrt{\pi}} \int_0^{\infty} e^{-t^2} dt =\frac{2}{\sqrt{\pi}} \times \frac{\sqrt{\pi}}{2} = 1 ; erf(0) =  \frac{2}{\sqrt{\pi}} \int_0^0 e^{-t^2} dt = 0 ; erf(-\infty)= -erf(\infty)=-1   \] which proves lemmas in \eqref{erfvalues}. 
\[ erfc(0) = \frac{2}{\sqrt{\pi}} \int_0^{\infty} e^{-t^2} dt =\frac{2}{\sqrt{\pi}} \times \frac{\sqrt{\pi}}{2} = 1 ; erfc(\infty) =  \frac{2}{\sqrt{\pi}} \int_{\infty}^{\infty} e^{-t^2} dt = 0 ; erfc(-\infty)= 2-erfc(\infty)=2-0=2   \] which proves lemmas in \eqref{erfcvalues}. 
\[ erfi(0) = \frac{2}{\sqrt{\pi}} \int_0^0 e^{t^2} dt =0   \] which proves lemma in \eqref{erfivalues}.  \end{proof}

\begin{theorem}
\begin{equation}
\int_0^{\infty} e^{-\ln(x)^2} dx = \sqrt[4]{e} \sqrt{\pi}
\end{equation}
\end{theorem}
\begin{proof}
\[ \int_0^{\infty} e^{-\ln(x)^2} dx \xrightarrow{ln(x)\rightarrow x } \int_{-\infty}^{\infty} e^{-x^2} e^x  dx = \int_{-\infty}^{\infty} e^{-(x-\frac{1}{2})^2+\frac{1}{4}} dx  \xrightarrow{(x-\frac{1}{2}) \rightarrow x} \sqrt[4]{e} \int_{-\infty}^{\infty}e^{-x^2} dx = \sqrt[4]{e} \sqrt{\pi}  \,\,   \]
\end{proof}

\begin{theorem}
\begin{equation}
 \int_0^{\infty} e^{-W(x)^2} dx = \sqrt[4]{e} \left[\frac{3\sqrt{\pi}}{4}+\frac{e^{-\frac{1}{4}}}{2}-\frac{3\sqrt{\pi}}{4}erf\left(-\frac{1}{2}\right)  \right]   
\end{equation}
where W(x) is Lambert W function, also known as the product logarithm function.
\end{theorem}
\begin{proof}
\[ \int_0^{\infty} e^{-W(x)^2} dx \xrightarrow{W(x)\rightarrow x } \int_{0}^{\infty} e^{-x^2} e^x(x+1)  dx = \int_{0}^{\infty} e^{-(x-\frac{1}{2})^2+\frac{1}{4}} (x-\frac{1}{2}+\frac{3}{2}) dx  \]
\[=\sqrt[4]{e} \left[ \frac{e^{-(x-\frac{1}{2})^2}}{-2} \Big|_0^{\infty} + \frac{3}{2} \int_0^{\infty} e^{-(x-\frac{1}{2})^2} dx  \right] = \sqrt[4]{e} \left[\frac{1}{2}e^{-\frac{1}{4}} + \frac{3}{2}\times \frac{\sqrt{\pi}}{2} erf(x-\frac{1}{2})\Big|_0^{\infty}   \right] \]
\[ = \sqrt[4]{e} \left[\frac{1}{2}e^{-\frac{1}{4}} + \frac{3\sqrt{\pi}}{4}-\frac{3\sqrt{\pi}}{4}erf(-\frac{1}{2}) \right] \,\,   \]
\end{proof}

Now,  let's consider more interesting functions for $f(x)$, including trigonometric, inverse trigonometric, and inverse hyperbolic functions. For trigonometric functions, we restrict the domain of integration from 0 to   $\frac{\pi}{2}  $. 

 \begin{theorem}
\begin{equation}\label{taneqn}
 \int_0^{\frac{\pi}{2}} e^{-\tan^2(x)} dx = \frac{e\pi}{2}erfc(1)  
\end{equation}
\end{theorem}
\begin{proof}
\[ \int_0^{\frac{\pi}{2}} e^{-\tan^2(x)} dx \xrightarrow{tan(x)\rightarrow x} \int_0^{\infty} \frac{e^{-x^2}}{1+x^2} dx = e \int_0^{\infty} \frac{e^{-(x^2+1)}}{1+x^2} dx = e\int_0^{\infty} \int_1^{\infty} e^{-(x^2+1) t} dt dx \]
Using Fubini's theorem to switch the order of integration
\[ e\int_1^{\infty} \int_0^{\infty} e^{-(x^2+1)t} dx dt = e\int_1^{\infty}e^{-t} \int_0^{\infty} e^{-x^2t} dx dt = e\int_1^{\infty}e^{-t} \frac{1}{2}\sqrt{\frac{\pi}{t}} dt = \frac{e\sqrt{\pi}}{2}\int_1^{\infty}e^{-t} t^{-\frac{1}{2}} dt \]
\[ \frac{e\sqrt{\pi}}{2}\int_1^{\infty}e^{-t} t^{-\frac{1}{2}} dt \xrightarrow{t\rightarrow t^2} e\sqrt{\pi} \int_1^{\infty} e^{-t^2} dt = e\sqrt{\pi} \times \frac{\sqrt{\pi}}{2} erfc(1) = \frac{e\pi}{2} erfc(1) \,\,   \]
\end{proof}

 \begin{theorem}
\begin{equation}\label{coteqn}
  \int_0^{\frac{\pi}{2}} e^{-\cot^2(x)} dx = \frac{e\pi}{2}erfc(1)  
\end{equation}
\end{theorem}
We establish a lemma before dealing with the theorem. 
\begin{lemma}
\begin{equation}\label{King}
 \int_a^b f(x) dx = \int_a^b f(a+b-x) dx 
\end{equation}
\end{lemma}
\begin{proof}
\[ \int_a^b f(x) dx \xrightarrow{x\rightarrow a+b-x} -\int_b^a f(a+b-x)  dx \xrightarrow{Reflection\,\, Property\,\,} \int_a^b f(a+b-x) dx \]
\end{proof}
Now, we are ready for the theorem \eqref{coteqn}.\\
\begin{proof}
\[ \int_0^{\frac{\pi}{2}} e^{-\cot^2(x)} dx \xrightarrow{ Eqn. \eqref{King}}\int_0^{\frac{\pi}{2}} e^{-\cot^2(\frac{\pi}{2}-x)} dx = \int_0^{\frac{\pi}{2}} e^{-tan^2(x)} dx \xrightarrow{Eqn. \eqref{taneqn}}= \frac{e\pi}{2} erfc(1)  \,\,   \]
\end{proof}

 \begin{theorem}
\begin{equation}\label{seceqn}
  \int_0^{\frac{\pi}{2}} e^{-\sec^2(x)} dx = \frac{\pi}{2}erfc(1)  
\end{equation}
\end{theorem}
\begin{proof}
\[ \int_0^{\frac{\pi}{2}} e^{-\sec^2(x)} dx \xrightarrow{\sec^2(x)=1+\tan^2(x)}\int_0^{\frac{\pi}{2}} e^{-1-tan^2(x)} dx =\frac{1}{e} \int_0^{\frac{\pi}{2}} e^{-tan^2(x)} dx \xrightarrow{Eqn. \eqref{taneqn}} = \frac{\pi}{2} erfc(1) \,\,   \]
\end{proof}

 \begin{theorem}
\begin{equation}\label{csceqn}
\int_0^{\frac{\pi}{2}} e^{-\csc^2(x)} dx = \frac{\pi}{2}erfc(1)  
\end{equation}
\end{theorem}
\begin{proof}
\[ \int_0^{\frac{\pi}{2}} e^{-\csc^2(x)} dx \xrightarrow{Eqn. \eqref{King}}\int_0^{\frac{\pi}{2}} e^{-\csc^2(\frac{\pi}{2}-x)} dx = \int_0^{\frac{\pi}{2}} e^{-sec^2(x)} dx \xrightarrow{Eqn. \eqref{seceqn}}= \frac{\pi}{2} erfc(1)  \,\,   \]
\end{proof}

 \begin{theorem}
\begin{equation}\label{sineqn}
 \int_0^{\frac{\pi}{2}} e^{-\sin^2(x)} dx =\frac{\pi}{2}e^{-\frac{1}{2}} I_0(\frac{1}{2})  
\end{equation}
where    $I_n(z)  $ is the modified Bessel function of first kind. 
\end{theorem}

Before we begin dealing with theorem   $\eqref{sineqn}  $, we cite the definition of the modified Bessel function of first kind from   \(\cite{W2002}\).

\begin{definition}
For integer n and   $Re(z)>0  $,
\begin{equation} \label{Besselfirst}
 I_n(z) = \frac{1}{\pi} \int_0^{\pi} e^{z \cos(\theta)} \cos(n\theta) \, d\theta 
\end{equation}
where   $I_n(z)  $ is the modified Bessel function of first kind.
\end{definition}
Now we are ready for the theorem \eqref{sineqn}. \\
\begin{proof}
\[ \int_0^{\frac{\pi}{2}} e^{-\sin^2(x)} dx \xrightarrow{1-\cos(2x)=2\sin^2(x)}  \int_0^{\frac{\pi}{2}} e^{-\frac{1}{2}} e^{\frac{\cos(2x)}{2}} dx \xrightarrow{2x \rightarrow x} \frac{1}{2}e^{-\frac{1}{2}} \int_0^{\pi} e^{\frac{1}{2}cos(x)} dx \]
\[ = \frac{1}{2}e^{-\frac{1}{2}} \int_0^{\pi} e^{\frac{1}{2}cos(x)} \cos(0.x) dx = \frac{\pi}{2}e^{-\frac{1}{2}} I_0\left(\frac{1}{2}\right)  \,\,   \]
\end{proof}

\begin{theorem}
\begin{equation}\label{coseqn}
 \int_0^{\frac{\pi}{2}} e^{-\cos^2(x)} dx =\frac{\pi}{2}e^{-\frac{1}{2}} I_0\left(\frac{1}{2}\right)   
\end{equation}
where    $I_n(z)  $ is the modified Bessel function of first kind.  \\
\end{theorem}
\begin{proof}
\[ \int_0^{\frac{\pi}{2}} e^{-\cos^2(x)} dx \xrightarrow{Eqn.  \eqref{King}}\int_0^{\frac{\pi}{2}} e^{-\cos^2(\frac{\pi}{2}-x)} dx = \int_0^{\frac{\pi}{2}} e^{-sin^2(x)} dx \xrightarrow{Eqn. \eqref{sineqn}}= \frac{\pi}{2}e^{-\frac{1}{2}} I_0(\frac{1}{2})  \,\,   \] 
\end{proof}

Next, we take $f(x)$ as inverse trigonometric functions. For inverse trigonometric functions, we restrict the domain of integration from 0 to 1.

 \begin{theorem}
\begin{equation}\label{arcsineqn}
 \int_0^{1} e^{-{\arcsin^2(x)}} dx = \frac{\sqrt{\pi} e^{-\frac{1}{4}}}{4} \left[ erfc(\frac{i}{2}) + erfc(-\frac{i}{2}) + i \left( erfi(\frac{1}{2}-\frac{i\pi}{2}) -erfi(\frac{1}{2}+\frac{i \pi}{2}) +2i \right)  \right] 
\end{equation}
\end{theorem}
\begin{proof}
\[ \int_0^{1} e^{-{\arcsin^2(x)}} dx \xrightarrow{x\rightarrow \sin(x) } \int_0^{\frac{\pi}{2}} e^{-x^2} \cos(x) dx \xrightarrow{Euler's\,\, Formula\,\, \cite{L2017}} \int_0^{\frac{\pi}{2}} e^{-x^2}(\frac{e^{ix}+e^{-ix}}{2}) dx \]
\[=\frac{1}{2} \int_0^{ \frac{\pi}{2}} \left( e^{-(x^2-ix)}+ e^{-(x^2+ix)} \right) dx = \frac{e^{-\frac{1}{4}}}{2} \int_0^{ \frac{\pi}{2}} \left( e^{-(x-\frac{i}{2})^2}+ e^{-(x+\frac{i}{2})^2} \right) dx \]

\[ =\frac{\sqrt{\pi}e^{-\frac{1}{4}}}{4}  \left( erf(x-\frac{i}{2}) + erf(x+\frac{i}{2}) \right) \Big|_0^{\frac{\pi}{2}} = \frac{\sqrt{\pi} e^{-\frac{1}{4}} }{4}  \left( erf(\frac{\pi}{2}-\frac{i}{2}) + erf(\frac{\pi}{2}+\frac{i}{2}) - erf(-\frac{i}{2})-erf(\frac{i}{2}) \right) \]
 Using  $erf(ix) = ierfi(x)$, $erf(-x) = -erf(x)$ and $erfc(x) = 1-erf(x)$ from \eqref{irelation}, \eqref{crelation}  and   \eqref{odderf},  
\[  = \frac{\sqrt{\pi} e^{-\frac{1}{4}} }{4}  \left( - i erfi(\frac{1}{2}+\frac{i \pi}{2}) +i erfi(\frac{1}{2}-\frac{i \pi}{2}) + erfc(-\frac{i}{2})+erfc(\frac{i}{2}) -2 \right)
\]
\[  = \frac{\sqrt{\pi} e^{-\frac{1}{4}} }{4}  \left[ erfc(\frac{i}{2}) + erfc(-\frac{i}{2}) + i \left( erfi(\frac{1}{2}-\frac{i \pi}{2}) -erfi(\frac{1}{2}+\frac{i \pi}{2})+2i \right) \right] \,\,  
\]
\end{proof}
\begin{theorem}
\begin{equation}\label{arccoseqn}
 \int_0^{1} e^{-{\arccos^2(x)}} dx = -\frac{\sqrt{\pi} e^{-\frac{1}{4}}}{4} \left[  erfi(\frac{1}{2}-\frac{i\pi}{2}) + erfi(\frac{1}{2}+\frac{i \pi}{2}) -2erfi(\frac{1}{2}) \right] 
\end{equation}
\end{theorem}
\begin{proof}
\[ \int_0^{1} e^{-{\arccos^2(x)}} dx \xrightarrow{x\rightarrow \cos(x) } \int_0^{\frac{\pi}{2}} e^{-x^2} \sin(x) dx \xrightarrow{Euler's  \,\, Formula \,\, \cite{L2017}} \int_0^{\frac{\pi}{2}} e^{-x^2}(\frac{e^{ix}-e^{-ix}}{2i}) dx \]
\[=\frac{1}{2i} \int_0^{ \frac{\pi}{2}} \left( e^{-(x^2-ix)}- e^{-(x^2+ix)} \right) dx = \frac{e^{-\frac{1}{4}}}{2i} \int_0^{ \frac{\pi}{2}} \left( e^{-(x-\frac{i}{2})^2}- e^{-(x+\frac{i}{2})^2} \right) dx \]

\[ =\frac{\sqrt{\pi}e^{-\frac{1}{4}}}{4i}  \left( erf(x-\frac{i}{2}) - erf(x+\frac{i}{2}) \right) \Big|_0^{\frac{\pi}{2}} = \frac{\sqrt{\pi} e^{-\frac{1}{4}} }{4i}  \left( erf(\frac{\pi}{2}-\frac{i}{2}) - erf(\frac{\pi}{2}+\frac{i}{2}) - erf(-\frac{i}{2})+erf(\frac{i}{2}) \right) \]
\end{proof}
  Using $erf(ix) = ierfi(x)$, $erf(-x) = - erf(x)$ from $\eqref{irelation}$  and   $\eqref{odderf}$, 
\[  = \frac{\sqrt{\pi} e^{-\frac{1}{4}} }{4i}  \left( - i erfi(\frac{1}{2}+\frac{i \pi}{2}) -i erfi(\frac{1}{2}-\frac{i \pi}{2}) + ierfi(\frac{1}{2})+ierfi(\frac{1}{2})  \right)
\]
\[  = -\frac{\sqrt{\pi} e^{-\frac{1}{4}} }{4}  \left( erfi(\frac{1}{2}+\frac{i\pi}{2}) + erfi(\frac{1}{2}-\frac{i\pi}{2}) - 2  erfi(\frac{1}{2}) \right) \,\,  
\] 

Next, we take $f(x)$ as inverse hyperbolic functions. For inverse hyperbolic functions, we let the domain of integration go from 0 to   $\infty  $  without any restriction. 

 \begin{theorem}
\begin{equation}\label{arcsinheqn}
 \int_0^{\infty} e^{-arcsinh^2(x)} dx = \frac{\sqrt{\pi}}{2} e^{\frac{1}{4}}   
\end{equation}
\end{theorem}
\begin{proof}
\[ \int_0^{\infty} e^{-arcsinh^2(x)} dx \xrightarrow[arcsinh(x)=\ln(x+\sqrt{1+x^2}) \cite{M2016}]{x\rightarrow \sinh(x)} \int_0^{\infty} e^{-x^2} \cosh(x) dx \xrightarrow{cosh(x)=\frac{e^x+e^{-x}}{2} \cite{MA2016}}  \int_0^{\infty} e^{-x^2}(\frac{e^{x}+e^{-x}}{2}) dx  \]

\[=\frac{1}{2} \int_0^{ \infty} \left( e^{-(x^2-x)}+ e^{-(x^2+x)} \right) dx = \frac{e^{\frac{1}{4}}}{2} \int_0^{ \infty} \left( e^{-(x-\frac{1}{2})^2}+ e^{-(x+\frac{1}{2})^2} \right) dx \]

\[ =\frac{\sqrt{\pi}e^{\frac{1}{4}}}{4}  \left( erf(x-\frac{1}{2}) + erf(x+\frac{1}{2}) \right) \Big|_0^{\infty} = \frac{\sqrt{\pi} e^{\frac{1}{4}} }{4}  \left( (1+1)-[erf(-\frac{1}{2})+erf(\frac{1}{2})] \right) \xrightarrow{using \, \eqref{odderf}} \frac{\sqrt{\pi}}{2}e^{\frac{1}{4}}    \]
\end{proof}

 \begin{theorem}
\begin{equation}\label{arccosheqn}
\int_0^{\infty} e^{-arccosh^2(x)} dx = \frac{\sqrt{\pi}}{4} e^{\frac{1}{4}} \left[ erf(\frac{1}{2}-\frac{i \pi}{2}) + erf(\frac{1}{2}+\frac{i \pi}{2}) \right]     
\end{equation}
\end{theorem}
\begin{proof}
\[ \int_0^{\infty} e^{-arccosh^2(x)} dx \xrightarrow[arccosh(x)=\ln(x+\sqrt{x^2-1}) \cite{M2016}]{x\rightarrow \cosh(x)} \int_{\frac{i \pi}{2}}^{\infty} e^{-x^2} \sinh(x) dx \xrightarrow{sinh(x)=\frac{e^x-e^{-x}}{2} \cite{MA2016}}  \int_{\frac{i\pi}{2}}^{\infty} e^{-x^2}(\frac{e^{x}-e^{-x}}{2}) dx  \]

\[=\frac{1}{2} \int_{\frac{i\pi}{2}}^{ \infty} \left( e^{-(x^2-x)}- e^{-(x^2+x)} \right) dx = \frac{e^{\frac{1}{4}}}{2} \int_{\frac{i\pi}{2}}^{ \infty} \left( e^{-(x-\frac{1}{2})^2}- e^{-(x+\frac{1}{2})^2} \right) dx \]

\[ =\frac{\sqrt{\pi}e^{\frac{1}{4}}}{4}  \left( erf(x-\frac{1}{2}) - erf(x+\frac{1}{2}) \right) \Big|_{\frac{i\pi}{2}}^{\infty} = \frac{\sqrt{\pi} e^{\frac{1}{4}} }{4}  \left( (1-1)-[erf(\frac{i\pi}{2}-\frac{1}{2})-erf(\frac{i\pi}{2}+\frac{1}{2})] \right) \]
\[ \xrightarrow{Using \, \eqref{odderf}} \frac{\sqrt{\pi} e^{\frac{1}{4}} }{4}  \left(erf(\frac{1}{2}-\frac{i \pi}{2})+erf(\frac{i\pi}{2}+\frac{1}{2}) \right)   \]
\end{proof}
\section{Gaussian Like Integral of Type-II}
In this section, we prove various results involving integrals in the form of   $\int_0^{\infty} e^{-x^2}f(x) dx  $. Most of the results are presented in terms of gamma function, $erf(x)$, $erfi(x)$ and $erfc(x)$. Using the Lemmas presented in the previous section, we are going to construct our main results. 

 \begin{theorem}
\begin{equation}
  \int_0^{\infty} e^{-x^2} x^n dx = \frac{1}{2}\Gamma(\frac{n+1}{2})   
\end{equation}
\end{theorem}
\begin{proof}
 \[ \int_0^{\infty} e^{-x^2} x^n dx \xrightarrow{x^2 \rightarrow x} \frac{1}{2}\int_0^{\infty} e^{-x}x^{\frac{n-1}{2}} dx  \xrightarrow{Using  \eqref{Gamma}} \frac{1}{2} \Gamma\left(\frac{n+1}{2}\right)  \]
\end{proof}

 \begin{theorem}
\begin{equation}
 \int_0^{\infty} e^{-x^2} \ln(x) dx = -\frac{\sqrt{\pi}}{4} \left(\gamma + \ln(4)\right)    
\end{equation}
where   $\gamma  $ is euler-mascheroni constant.
\end{theorem}
\begin{proof}
 \[ \int_0^{\infty} e^{-x^2} ln(x) dx \xrightarrow{x^2 \rightarrow x} \frac{1}{2}\int_0^{\infty} e^{-x} \ln(x^{\frac{1}{2}}) x^{-\frac{1}{2}} dx =  \frac{1}{4}\int_0^{\infty} e^{-x} \ln(x) x^{\frac{1}{2}-1} dx . \]
 But we have   $ \Gamma(n) = \int_0^{\infty} e^{-x}x^{n-1} dx  $ from   $\cite{SG2002}.  $ So,   $ \Gamma'(n) = \int_0^{\infty} e^{-x} x^{n-1} \ln(x) dx .  $ Thus required integral is   $\frac{1}{4} \Gamma'(\frac{1}{2})  $. From   $\cite{W2002}  $, we have   $\Gamma'(x)=\Gamma(x)\psi(x)  $. So, the result is   $\frac{1}{4} \Gamma(\frac{1}{2}) \psi(\frac{1}{2})  $. Again, from   $\cite{A2010}  $, we have   $\Gamma(\frac{1}{2})=\sqrt{\pi}  $ and   $\psi(\frac{1}{2})=(-\gamma-2\ln(2)).  $ Thus, the result   $-\frac{\sqrt{\pi}}{4}(\gamma+2\ln(2))$ is proved.    \\
\end{proof}

 Next, we take $f(x)$ as trigonometric functions without restricting the domain of integration. 
 
 \begin{theorem}
\begin{equation}
 \int_0^{\infty} e^{-x^2} cos(x) dx = \frac{\sqrt{\pi}}{2} e^{-\frac{1}{4}}     
\end{equation}
\end{theorem}
\begin{proof}
 \[ \int_0^{\infty} e^{-x^2} cos(x) dx \xrightarrow{ Euler's\,\, Formula \,\, \cite{L2017}}  \int_0^{\infty} e^{-x^2}(\frac{e^{ix}+e^{-ix}}{2}) dx  \]
 \[=\frac{1}{2} \int_0^{ \infty} \left( e^{-(x^2-ix)}+ e^{-(x^2+ix)} \right) dx = \frac{e^{-\frac{1}{4}}}{2} \int_0^{ \infty} \left( e^{-(x-\frac{i}{2})^2}+ e^{-(x+\frac{i}{2})^2} \right) dx \]

\[ =\frac{\sqrt{\pi}e^{-\frac{1}{4}}}{4}  \left( erf(x-\frac{i}{2}) + erf(x+\frac{i}{2}) \right) \Big|_0^{\infty} \] 
From \eqref{odderf} and   \eqref{erfvalues}, we have $erf(-x) = -erf(x)$  and     $erf(\infty)  = 1$ 
 \[\frac{\sqrt{\pi} e^{-\frac{1}{4}} }{4}  \left( 1+ 1 - erf(-\frac{i}{2})-erf(\frac{i}{2}) \right) = \frac{\sqrt{\pi} e^{-\frac{1}{4}} }{4}  \left( 2 + erf(\frac{i}{2})-erf(\frac{i}{2}) \right) = \frac{\sqrt{\pi} e^{-\frac{1}{4}} }{2}   \]
\end{proof}
 \begin{theorem}
\begin{equation}
 \int_0^{\infty} e^{-x^2} sin(x) dx = \frac{\sqrt{\pi}}{2} e^{-\frac{1}{4}} erfi(\frac{1}{2})    
\end{equation}
\end{theorem}
\begin{proof}
 \[ \int_0^{\infty} e^{-x^2} sin(x) dx \xrightarrow{ Euler's\,\, Formula\,\, \cite{L2017}}  \int_0^{\infty} e^{-x^2}(\frac{e^{ix}-e^{-ix}}{2i}) dx  \]
 \[=\frac{1}{2i} \int_0^{ \infty} \left( e^{-(x^2-ix)}- e^{-(x^2+ix)} \right) dx = \frac{e^{-\frac{1}{4}}}{2i} \int_0^{ \infty} \left( e^{-(x-\frac{i}{2})^2}- e^{-(x+\frac{i}{2})^2} \right) dx \]

\[ =\frac{\sqrt{\pi}e^{-\frac{1}{4}}}{4i}  \left( erf(x-\frac{i}{2}) - erf(x+\frac{i}{2}) \right) \Big|_0^{\infty} \] 
From \eqref{irelation},\eqref{odderf} and   \eqref{erfvalues}, we have $erf(ix) = ierfi(x)$, $erf(-x) = -erf(x)$ and $erf(\infty) = 1 $ 
 \[\frac{\sqrt{\pi} e^{-\frac{1}{4}} }{4i}  \left( 1- 1 - erf(-\frac{i}{2})+erf(\frac{i}{2}) \right) = \frac{\sqrt{\pi} e^{-\frac{1}{4}} }{4i}\times 2 ierfi(\frac{1}{2}) = \frac{\sqrt{\pi} e^{-\frac{1}{4}} }{2} erfi(\frac{1}{2})   \]
\end{proof}
 Next, we take $f(x)$ as hyperbolic functions without restricting the domain of integration. 
 
 \begin{theorem}
\begin{equation}
 \int_0^{\infty} e^{-x^2} cosh(x) dx = \frac{\sqrt{\pi}}{2} e^{\frac{1}{4}}    
\end{equation}
\end{theorem}
\begin{proof}
 \[ \int_0^{\infty} e^{-x^2} cosh(x) dx \xrightarrow{ \cosh(x)=\frac{e^x+e^{-x}}{2} \cite{MA2016}}  \int_0^{\infty} e^{-x^2}(\frac{e^{x}+e^{-x}}{2}) dx  \]
 \[=\frac{1}{2} \int_0^{ \infty} \left( e^{-(x^2-x)}+ e^{-(x^2+x)} \right) dx = \frac{e^{\frac{1}{4}}}{2} \int_0^{ \infty} \left( e^{-(x-\frac{1}{2})^2}+ e^{-(x+\frac{1}{2})^2} \right) dx \]

\[ =\frac{\sqrt{\pi}e^{\frac{1}{4}}}{4}  \left( erf(x-\frac{1}{2}) + erf(x+\frac{1}{2}) \right) \Big|_0^{\infty} \] 
From \eqref{odderf} and   \eqref{erfvalues}, we have erf(-x)=-erf(x)  and     $erf(\infty)  $=1 
 \[\frac{\sqrt{\pi} e^{\frac{1}{4}} }{4}  \left( 1+ 1 - erf(-\frac{1}{2})-erf(\frac{1}{2}) \right) = \frac{\sqrt{\pi} e^{\frac{1}{4}} }{4}  \left( 2 + erf(\frac{1}{2})-erf(\frac{1}{2}) \right) = \frac{\sqrt{\pi} e^{\frac{1}{4}} }{2}   \]
\end{proof}
 \begin{theorem}
\begin{equation}
 \int_0^{\infty} e^{-x^2} sinh(x) dx = \frac{\sqrt{\pi}}{2} e^{\frac{1}{4}} erf(\frac{1}{2})    
\end{equation}
\end{theorem}
\begin{proof}
 \[ \int_0^{\infty} e^{-x^2} sinh(x) dx \xrightarrow{ \sinh(x)=\frac{e^x+e^{-x}}{2} \cite{MA2016}}  \int_0^{\infty} e^{-x^2}(\frac{e^{x}-e^{-x}}{2}) dx  \]
 \[=\frac{1}{2} \int_0^{ \infty} \left( e^{-(x^2-x)}- e^{-(x^2+x)} \right) dx = \frac{e^{\frac{1}{4}}}{2} \int_0^{ \infty} \left( e^{-(x-\frac{1}{2})^2}- e^{-(x+\frac{1}{2})^2} \right) dx \]

\[ =\frac{\sqrt{\pi}e^{-\frac{1}{4}}}{4}  \left( erf(x-\frac{1}{2}) - erf(x+\frac{1}{2}) \right) \Big|_0^{\infty} \] 
From \eqref{odderf} and   \eqref{erfvalues}, we have erf(-x)=-erf(x)  and     $erf(\infty)  $=1 
 \[\frac{\sqrt{\pi} e^{-\frac{1}{4}} }{4}  \left( 1- 1 - erf(-\frac{1}{2})+erf(\frac{1}{2}) \right) = \frac{\sqrt{\pi} e^{-\frac{1}{4}} }{4}\times 2 erf(\frac{1}{2}) = \frac{\sqrt{\pi} e^{-\frac{1}{4}} }{2} erf(\frac{1}{2})   \]
 \end{proof}

 To make this more interesting, next, we take $f(x)$ as error functions themselves letting the domain of integration from 0 to   $\infty $.
 
\begin{theorem}
\begin{equation}
\int_0^{\infty} e^{-x^2} erf(x) dx = \frac{\sqrt{\pi}}{4}   
\end{equation}
\end{theorem}
\begin{proof}
 \[ \int_0^{\infty} e^{-x^2} erf(x) dx \xrightarrow{From \, \eqref{erf}} \frac{2}{\sqrt{\pi}} \int_0^{\infty} e^{-x^2} \int_0^x e^{-t^2} dt dx \xrightarrow{\frac{t}{x}\rightarrow t} \frac{2}{\sqrt{\pi}} \int_0^{\infty} e^{-x^2} \int_0^1 e^{-(xt)^2} x dt dx \]
 \[ =\frac{2}{\sqrt{\pi}} \int_0^{\infty}\int_0^1 e^{-x^2(1+t^2)}x dt dx \xrightarrow{Using\,\, Fubini's\,\, theorem} \frac{1}{\sqrt{\pi}}\int_0^1\int_0^{\infty} e^{-x^2(1+t^2)}2x dx dt\]
 \[= \frac{1}{\sqrt{\pi}} \int_0^1 \frac{e^{-x^2(1+t^2)}}{-(1+t^2)} \Big|_0^{\infty} dt = \frac{1}{\sqrt{\pi}} \int_0^1 \frac{1}{1+t^2} dt = \frac{1}{\sqrt{\pi}} \arctan(t)\Big|_0^1 = \frac{1}{\sqrt{\pi}} \times \frac{\pi}{4}= \frac{\sqrt{\pi}}{4}    \]
 \end{proof}
 \begin{theorem}
\begin{equation}
 \int_0^{\infty} e^{-x^2} erfc(x) dx = \frac{\sqrt{\pi}}{4}    
\end{equation}
\end{theorem}
\begin{proof}
 \[ \int_0^{\infty} e^{-x^2} erfc(x) dx \xrightarrow{From \, \eqref{erf}} \frac{2}{\sqrt{\pi}} \int_0^{\infty} e^{-x^2} \int_x^{\infty} e^{-t^2} dt dx \xrightarrow{\frac{t}{x}\rightarrow t} \frac{2}{\sqrt{\pi}} \int_0^{\infty} e^{-x^2} \int_1^{\infty} e^{-(xt)^2} x dt dx \]
 \[ =\frac{2}{\sqrt{\pi}} \int_0^{\infty}\int_1^{\infty} e^{-x^2(1+t^2)}x dt dx \xrightarrow{Using\,\, Fubini's\,\, theorem} \frac{1}{\sqrt{\pi}}\int_1^{\infty}\int_0^{\infty} e^{-x^2(1+t^2)}2x dx dt\]
 \[= \frac{1}{\sqrt{\pi}} \int_1^{\infty} \frac{e^{-x^2(1+t^2)}}{-(1+t^2)} \Big|_0^{\infty} dt = \frac{1}{\sqrt{\pi}} \int_1^{\infty} \frac{1}{1+t^2} dt = \frac{1}{\sqrt{\pi}} \arctan(t)\Big|_1^{\infty} = \frac{1}{\sqrt{\pi}} \times \frac{\pi}{4}= \frac{\sqrt{\pi}}{4}    \]
\end{proof}
 \section{Miscellaneous Results on other Generalizations of Gaussian Integrals}
 For the sake of completeness, we also evaluate another kind of generalized Gaussian Integral in this paper. 
 \textbf{Remarks: }
 For   $a>0  $
 \begin{equation}\label{general}
 \int_0^{\infty} e^{-(ax^2+bx+c)} dx= \frac{e^{\frac{b^2-4ac}{4a}}}{\sqrt{a}} erfc\left({\frac{b\sqrt{a}}{2a}}\right)  
 \end{equation}

 \begin{proof}
  \[ \int_0^{\infty} e^{-(ax^2+bx+c)} dx = \int_0^{\infty} e^{-[a((x+\frac{b}{2a})^2+ \frac{c}{a}- \frac{b^2}{4a^2})]} dx = \int_0^{\infty} e^{-a((x+\frac{b}{2a})^2+ \frac{4ac-b^2}{4a^2})} dx \]
  \[  = e^{\frac{b^2-4ac}{4a}}\int_0^{\infty} e^{-a(x+\frac{b}{2a})^2} dx \xrightarrow{(x+\frac{b}{2a})\rightarrow x}  e^{\frac{b^2-4ac}{4a}}\int_{\frac{b}{2a}}^{\infty} e^{-ax^2} dx \xrightarrow{\sqrt{a}x\rightarrow x} \frac{e^{\frac{b^2-4ac}{4a}}}{\sqrt{a}}\int_{\frac{b\sqrt{a}}{2a}}^{\infty} e^{-x^2} dx     \]
  \[ = \frac{ \sqrt{\pi}e^{\frac{b^2-4ac}{4a}}}{2\sqrt{a}} erf(x)\Big|_{\frac{b\sqrt{a}}{2a}}^{\infty}  =  \frac{\sqrt{\pi}e^{\frac{b^2-4ac}{4a}}}{2\sqrt{a}} erfc({\frac{b\sqrt{a}}{2a}})   \]
  \end{proof}
  Substituting $b = c = 0$ in \eqref{general} gives a well-known special case. \\
  \begin{equation}\label{specialgeneral}
   \int_0^{\infty} e^{-ax^2} dx = \frac{\sqrt{\pi}e^0}{2\sqrt{a}} erfc(0) = \frac{1}{2} \sqrt{\frac{\pi}{a}} 
  \end{equation}

\section{Conclusions} In this paper, we have systematically analyzed and evaluated several generalized Gaussian integrals and Gaussian-like integrals of types I and II. By leveraging specialized functions such as the error function, complementary error function, and imaginary error function, we derived a variety of useful results that extend the classical Gaussian integral's scope. These integrals demonstrate applications in diverse fields, including probability theory, statistical mechanics, and quantum mechanics, thereby emphasizing their profound mathematical and practical significance.

Our study also highlighted the versatility of Gaussian-related integrals when generalized to involve continuous functions or higher-order powers. The mathematical techniques employed here, including real and complex analysis, are crucial for further exploration of these integrals and their applications.

We encourage interested readers to explore these results further and investigate additional integrals. For instance, one might consider cases where the function is gamma function $f(x) =  \Gamma(x)$, digamma function $f(x) = \psi(x)$, zeta function $f(x) = \zeta(x)$, poly-logarithm function $f(x) = Li_2(x) $ or any other special functions within the Gaussian-Like integrals.The results presented here serve as a foundation for such endeavors, offering a robust framework for both theoretical investigations and practical implementations.

\end{document}